\documentclass[11pt]{amsart}
\usepackage{amscd,amsxtra,amssymb,mathrsfs, bbm}
\usepackage{stmaryrd}

\usepackage{geometry}
    \newenvironment{dedication}
        {\begin{quotation}\begin{center}\begin{em}}
        {\par\end{em}\end{center}\end{quotation}}

\newtheorem{theorem}{Theorem}[section]
\newtheorem{corollary}[theorem]{Corollary}
\newtheorem{lemma}[theorem]{Lemma}
\newtheorem{proposition}[theorem]{Proposition}

\theoremstyle{definition}
\newtheorem{definition}[theorem]{Definition}
\newtheorem{remark}[theorem]{Remark}
\newtheorem{example}[theorem]{Example}

\DeclareMathOperator{\Hom}{\mathsf{Hom}}

\DeclareMathOperator{\Ext}{\mathsf{Ext}}

\DeclareMathOperator{\End}{\mathsf{End}}

\DeclareMathOperator{\Spec}{\mathsf{Spec}}

\newcommand{\kk}{\mathbbm{k}}

\newcommand{\kEnd}{\mbox{\textit{End}}}
\newcommand{\kHom}{\mbox{\textit{Hom}}}
\newcommand{\kExt}{\mbox{\textit{Ext}}}
\newcommand{\kMat}{\mbox{\textit{Mat}}}

\def\kA{\mathcal A} 
\def\kB{\mathcal B} \def\kO{\mathcal O}
\def\kC{\mathcal C} \def\kP{\mathcal P}
 \def\kQ{\mathcal Q}
 
\def\kF{\mathcal F} \def\kS{\mathcal S}
\def\kG{\mathcal G} \def\kT{\mathcal T}
\def\kH{\mathcal H} 
\def\kI{\mathcal I} 
\def\kJ{\mathcal J} 
\def\kK{\mathcal K} \def\kX{\mathcal X}
 \def\kY{\mathcal Y}

\newcommand{\bbX}{\mathbb{X}}

 \def\sP{\mathsf P}
\def\sD{\mathsf D} 
 \def\sR{\mathsf R}
\def\sF{\mathsf F} 
\def\sG{\mathsf G} 
\def\sH{\mathsf H} 
\def\sI{\mathsf I} 
\def\sJ{\mathsf J} 
 
\def\sL{\mathsf L}

\def\al{\alpha}       
\def\be{\beta}        
\def\ga{\gamma}

\def\kron#1#2{\xymatrix@C=2em{{#1}\ar@/^3pt/[r]\ar@/_3pt/[r]&{#2}}}
\def\bu{{\scriptscriptstyle\bullet}}
\def\xarr{\xrightarrow}
\def\smtr#1{\left(\begin{smallmatrix}#1\end{smallmatrix}\right)}

\input xy
\xyoption{all}

\title[Singular curves and quasi--hereditary algebras]{Singular curves and quasi--hereditary algebras}

\author{Igor Burban}
\address{
Universit\"at zu K\"oln,
Mathematisches Institut,
Weyertal 86-90,
D-50931 K\"oln,
Germany
}
\email{burban@math.uni-koeln.de}

\author{Yuriy Drozd}
\address{
 Institute of Mathematics\\
National Academy of Sciences of Ukraine,
Tereschenkivska str. 3,
01004 Kyiv, Ukraine}
\email{drozd@imath.kiev.ua, y.a.drozd@gmail.com}
\urladdr{www.imath.kiev.ua/$\sim$drozd}

\author{Volodymyr Gavran}
\address{
 Institute of Mathematics\\
National Academy of Sciences of Ukraine,
Tereschenkivska str. 3,
01004 Kyiv, Ukraine
}
\email{vlgvrn@gmail.com}


\begin{document}
\maketitle

\begin{dedication}
\vspace*{1mm}{To the memory of Sergiy Ovsienko}
\end{dedication}

\noindent
\newcommand{\Perf}{\mathop{\mathsf{Perf}}\nolimits}
\newcommand{\Coh}{\mathop{\mathsf{Coh}}\nolimits}
\newcommand{\Qcoh}{\mathop{\mathsf{Qcoh}}\nolimits}
\newcommand{\lar}{\longrightarrow}
\newcommand{\idm}{\mathfrak{m}}

\begin{abstract}
In this article  we construct a categorical resolution of singularities of an excellent reduced curve  $X$, introducing a certain sheaf of orders on $X$. This categorical resolution is shown to be  a recollement of the derived category of coherent sheaves on the normalization of $X$ and the derived category of finite length modules over a certain artinian quasi--hereditary  ring $Q$ depending purely on the local singularity types of $X$.

Using this technique, we prove several statements on the Rouquier
dimension of the derived category of coherent sheaves on $X$. Moreover, in the case $X$ is rational and projective we construct  a finite dimensional quasi--hereditary algebra $\Lambda$ such that the triangulated category
$\Perf(X)$ embeds into $D^b(\Lambda-\mathsf{mod})$ as a full subcategory.
\end{abstract}
\section{Introduction}
Let $X$ be a curve, $\widetilde{X} \stackrel{\nu}\rightarrow X$ its normalization, $\kO = \kO_X$  and $\widetilde\kO =\nu_*(\kO_{\widetilde{X}})$.
 Generalizing an original  idea of K\"onig \cite{koe},  we define a sheaf of orders $\kA$ on $X$  called \emph{K\"onig's order} such that the ringed space
$\bbX = (X, \kA)$ has the following properties.

\medskip
\noindent
1.~The non--commutative curve $\bbX$ is ``smooth'' in the sense that \linebreak
$\mathrm{gl.dim}\bigl(\Coh(\bbX)\bigr) < \infty$, where $\Coh(\bbX)$ is the category of coherent $\kA$--modules on $X$.  In fact, $\mathrm{gl.dim}\bigl(\Coh(\bbX)\bigr) \le 2n$, where $n$ is a certain (purely commutative) invariant of $X$ called \emph{level}. If the original curve  $X$ has only nodes and cusps as singularities, the sheaf $\kA$ coincides with \emph{Auslander's order}
    $$
    \left(
    \begin{array}{cc}
    \kO & \widetilde\kO \\
    \kI & \widetilde\kO
    \end{array}
    \right)
    $$
  introduced in \cite{bd}, where $\kI$ is the ideal sheaf of the singular locus of $X$.

\medskip
\noindent
2.~The non--commutative curve $\bbX$ is a \emph{non--commutative} (or \emph{categorical}) \emph{resolution of singularities} of $X$, see  \cite{vdb,ku} for the definitions.
The category $\Coh(X)$ of coherent sheaves on $X$ is a Serre quotient of $\Coh(\bbX)$. Moreover, the triangulated category $\Perf(X)$ of perfect complexes on $X$ admits  an exact fully faithful embedding $\Perf(X) \hookrightarrow D^b\bigl(\Coh(\bbX)\bigr)$ such that its composition with the Verdier localization $D^b\bigl(\Coh(\bbX)\bigr) \rightarrow D^b\bigl(\Coh(X)\bigr)$ is isomorphic to the canonical inclusion functor.
If the curve $X$ is Gorenstein, the constructed  categorical resolution of singularities of $X$ turns out to be  \emph{weakly crepant} in the sense of Kuznetsov \cite{ku}.

\medskip
\noindent
3.~We show that the triangulated category $D^b\bigl(\Coh(\bbX)\bigr)$
is a recollement of $D^b\bigl(\Coh(\widetilde{X})\bigr)$ and $D^b(Q-\mathrm{mod})$,
where $Q$ is a certain \emph{quasi-hereditary} artinian ring (in particular, of finite global dimension), determined ``locally'' by the singularity types  of the singular points of $X$. In the case of simple curve singularities, we describe the corresponding algebras $Q$
explicitly in terms of quivers and relations.

    \medskip
\noindent
4.~Assume $X$ is projective over some field $\kk$. According to Orlov \cite{or}, the Rouquier dimension \cite{rou} of the triangulated category
$D^b\bigl(\Coh(\widetilde{X})\bigr)$ is equal to \emph{one}. Let $\widetilde\kF$ be a vector bundle on $\tilde{X}$ such that
$\langle\widetilde{\kF}\rangle_2 =   D^b\bigl(\Coh(\widetilde{X})\bigr)$ and $\kF = \nu_*(\widetilde\kF)$. We show that
$
D^b\bigl(\Coh(X)\bigr) = \bigl\langle \kF \oplus \kO_Z\bigr\rangle_{n+2}
$
where $\kO_Z$ is the structure sheaf of the singular locus of $X$ (with respect to the reduced scheme structure) and
$n$ is the level of $X$.

\medskip
\noindent
5.~If our original curve $X$ is moreover  rational, then we show that $D^b\bigl(\Coh(\bbX)\bigr)$ admits a
\emph{tilting object} $\kH$ such that the finite dimensional $\kk$--algebra $\Lambda = \bigl(\End_{D^b(\bbX)}(\kH)\bigr)^{\mathrm{op}}$ is  quasi--hereditary. In particular, we get an exact  fully faithful embedding
    $\Perf(X) \hookrightarrow D^b(\Lambda-\mathrm{mod})$,
    giving an affirmative answer on  a question posed to the first--named author by Valery Lunts.

\medskip
\noindent
\emph{Acknowledgement}. The work on this article has been started  during the stay of the second--named author at the Max--Planck--Institut f\"ur Mathematik in Bonn. Its final version was prepared during the visit of the second-- and the third--named author to the Institute of Mathematics of the University of Cologne. The first--named author would like to thank Valery Lunts for the invitation and iluminative discussions during his visit to the Indiana University Bloomington. We  are thankful to the referees for their  helpful comments.

\section{Local description of K\"onig's order}\label{Sec:LocalStory}

Let $(O, \idm)$ be a reduced Noetherian local ring of Krull dimension one, $K$ be its total ring of fractions and
$\widetilde{O}$ be the normalization of $O$.

\begin{proposition}\label{P:normalizelocal} Consider the ring $O^\sharp = \End_{O}(\idm)$. Then the following properties hold.
\begin{itemize}
\item $O^\sharp \cong \left\{x \in K \big| x \idm \subset \idm \right\}$. Moreover,
$O \subseteq O^\sharp \subseteq \widetilde{O}$ and $O  =  O^\sharp$ if and only if $O$ is regular.
\item Assume that $O$ is not regular. Then the canonical morphisms of $O$--modules
$$
\idm \stackrel{\varphi}\lar \Hom_{O}(O^\sharp, O) \quad \mbox{and} \quad
O^\sharp \stackrel{\psi}\lar \Hom_{O}(\idm, O)
$$
are isomorphisms.
\end{itemize}
\end{proposition}

\begin{proof}
For the first part,  see for example  \cite[Proposition 4]{com} or  \cite[Theorem 1.5.13]{deJongPfister}. To show the second part, note that $\varphi$ assigns
to an element $a \in \idm$ a morphism $O^\sharp \stackrel{\varphi_a}\lar O$, where $\varphi_a(x) = ax$. It is clear that $\varphi$ is injective. Since $\Hom_{O}(O^\sharp, O)$  viewed as a subset of $K$ is a proper  ideal in $O$, it is contained in $\idm$.
Hence, $\varphi$
 is also surjective, hence bijective.

Next, the canonical morphism $\Hom_O(\idm, \idm) \stackrel{\psi}\lar \Hom_O(\idm, O)$ is injective. On the other hand, there are  no surjective morphisms $\idm \rightarrow O$ (otherwise, $O$ would be a discrete valuation domain), hence the image of any morphism $\idm \rightarrow O$ belongs to $\idm$ and $\psi$ is surjective.
\end{proof}

From now on, let $O$ be an \emph{excellent} reduced Noetherian ring of Krull dimension one (see for example
\cite[Section 8.2]{liu} for the definition and main properties of excellent rings). As before, $K$ denotes its total ring of fractions and
$\widetilde{O}$ is the normalization of $O$. Let $X = \Spec(O)$ and $Z$ be the singular locus of $X$ equipped
with the reduced scheme structure. In other words,
$$Z =
\bigl\{\idm_1, \dots, \idm_t\bigr\}
= \bigl\{\idm \in \Spec(O) \; \big| \; O_\idm \; \mbox{is not regular}\bigr\}
$$
(the condition that $O$ is excellent implies that $Z$ is indeed a finite set).

\begin{proposition}\label{P:DetailsKoenigResol} Let
$
I = I_Z = \idm_1 \cap \dots \cap \idm_t
$
be the vanishing  ideal  of $Z$ and $O^\sharp = \End_{O}(I)$. Then the  following properties are true.
\begin{itemize}
\item $O^\sharp \cong \left\{x \in K \big| x \idm \subset \idm \right\}$. Moreover,
$O \subseteq O^\sharp \subseteq \widetilde{O}$ and $O  =  O^\sharp$ if and only if $O$ is regular.
\item Assume that $O$ is not regular. Then the canonical morphisms of $O$--modules
$
\idm \stackrel{\varphi}\lar \Hom_{O}(O^\sharp, O)$ and
$O^\sharp \stackrel{\psi}\lar \Hom_{O}(\idm, O)$
are isomorphisms.
\end{itemize}
\end{proposition}

\begin{proof}
For the first part, see again \cite[Proposition 4]{com} or  \cite[Theorem 1.5.13]{deJongPfister}. To prove the second, observe that the maps $\varphi$ and $\psi$ are well--defined and compatible with localizations with respect to a maximal ideal.
Hence, Proposition \ref{P:normalizelocal} implies the claim.
\end{proof}

We define a sequence of overrings $O_i$ of the initial ring $O$ by the following recursive procedure:
\begin{itemize}
\item $O_1 = O$.
\item $O_{i+1} = O_i^\sharp$ for $i \ge 1$.
\end{itemize}
Since the ring $O$ is excellent, the normalization $\widetilde{O}$ is finite over $O$, see for example \cite[Theorem 6.5]{ddone} or \cite[Section 8.2]{liu}. Hence,
there exists $n \in \mathbb{N}$ (called the \emph{level} of $O$) such that we have a finite chain of overrings
$$
O_1 \subset O_2 \subset \dots \subset O_n \subset O_{n+1}
$$
with $O_1 = O$ and
$O_{n+1} = \widetilde{O}$.

\begin{definition}\label{D:KoenigsOrder}
The ring $A := \End_{O}(O_1 \oplus O_2 \oplus \dots \oplus O_{n+1})^{\mathrm{op}}$ is called the \emph{K\"onig's order} of $O$.
\end{definition}

\begin{proposition}\label{P:KoenigOrderMatrDescr}
For any $1 \le i, j \le n+1$ pose $A_{ij} := \Hom_{O}(O_i, O_j)$. Then the following properties are true.
\begin{itemize}
\item For $i \le j$ we have: $A_{ij} \cong O_j$.
\item For $i > j$ we have: $A_{ij} \cong I_{i,j} := \Hom_{O_j}(O_i, O_j)$. In particular,  $I_{n+1,1} \cong C := \Hom_{O}(\widetilde{O}, O)$
is the \emph{conductor ideal}.
\item Next, $I_i:= I_{i+1,i}$ is the ideal of the singular locus of $\Spec(O_i)$ and
the ring $\bar{O}_i := O_i/I_i$ is semi--simple.
\item Moreover, the ideal  $I_{n+1, k}$ is projective over  $\kO_{n+1}$ for any $1 \le k \le n$.
\item The ring $A$  admits  the following ``matrix description'':
\begin{equation}\label{E:KoenigOrder}
A \cong
\left(
\begin{array}{llccc}
O_1 & O_2 & \dots & O_n & O_{n+1} \\
I_1 & O_2 & \dots & O_n & O_{n+1} \\
\vdots & \vdots & \ddots & \vdots & \vdots \\
I_{n,1} & I_{n, 2} & \dots & O_n & O_{n+1} \\
I_{n+1,1} & I_{n+1, 2} & \dots & I_n & O_{n+1} \\
\end{array}
\right)
\end{equation}
and $A \otimes_O K \cong \mathsf{Mat}_{n+1, n+1}(K)$. In other words, $A$ is an \emph{order} in the semi-simple algebra
$\mathsf{Mat}_{n+1, n+1}(K)$.
\item For any $2 \le i \le n+1$ and $1 \le j \le n$ we have inclusions
\begin{itemize}
\item $I_{i,1} \subset I_{i,2} \subset \dots \subset I_{i,i-1} \subset O_i \subset \dots \subset O_{n+1}$
\item $I_{n+1,j} \subset I_{n,j} \subset \dots \subset I_{j+1,j} \subset O_j$
\end{itemize}
describing the ``hierarchy'' between  the entries in every row and every column in the matrix description
(\ref{E:KoenigOrder}) of the ring $A$.
\end{itemize}
\end{proposition}

\begin{proof}
We have the following canonical isomorphisms of $O$--modules:
$$
O_j \cong \Hom_{O_i}(O_i, O_j) \stackrel{\cong}\lar \Hom_{O}(O_i, O_j)
$$
provided $i \le j$ as well as
$$
I_{i,j} := \Hom_{O_j}(O_i, O_j) \stackrel{\cong}\lar \Hom_{O}(O_i, O_j)
$$
for $i > j$. Proposition \ref{P:DetailsKoenigResol} implies that the ideal $I_i = I_ {i+1, i}$ is  indeed the ideal of the singular locus of $\Spec(O_i)$, hence the quotient $\bar{O}_i = O_i/I_i$ is semi--simple.
Since the ring $O_{n+1}$ is regular and the ideal  $I_{n+1,k}$ is torsion free as $O_{n+1}$--module, it is projective over $O_{n+1}$.

Finally, for any $1 \le j \le n$ and $1 \le i \le n+1$ the inclusion $O_j \subset O_{j+1}$ induces embeddings of $O$--modules
\begin{align*}
 \Hom_{O}(O_{j+1}, O_i)& \hookrightarrow \Hom_{O}(O_{j}, O_i)\\
 \intertext{and}
 \Hom_{O}(O_{i}, O_j)&\hookrightarrow \Hom_{O}(O_{i}, O_{j+1}).
\end{align*}
\vskip-1.5em
\end{proof}

\begin{remark}
The idea to study such a ring $A$ is due to K\"onig \cite{koe}, who considered a similar but slightly different construction.
\end{remark}

\smallskip
\noindent
For any $1 \le i \le n+1$ let  $e_i = e_{i,i}$ be the $i$-th standard idempotent of $A$ with respect to the presentation (\ref{E:KoenigOrder}). For $1 \le k \le n$ we denote
\begin{itemize}
\item $\varepsilon_{k} := \sum\limits_{i = k+1}^{n+1} e_i$, $J_k := A \varepsilon_k A$ and
$Q_k := A/J_k$.
\item In what follows we write $e = e_{n+1}$, $J = A e A = J_{n+1}$ and $Q := A/J = Q_n$.
\end{itemize}

\begin{theorem}\label{T:GlobalDimensionKoenigOrder} The global dimension of $A$ is finite:
 $\mathrm{gl.dim}(A) \le 2n$. Moreover, the artinian ring $Q = Q_O$ is quasi--hereditary (hence, its global dimension is finite, too).
\end{theorem}

\begin{proof} A straightforward calculation shows that for every $2 \le k \le n+1$ the two--sided ideal $J_{k-1}$ has the following matrix description:
$$
J_{k-1} =
\left(
\begin{array}{llclllcc}
I_{k,1} & I_{k,2} & \dots &I_{k,k-1} & O_k & O_{k+1} & \dots & O_{n+1} \\
  \vdots& \vdots & \vdots &\vdots & \vdots& \vdots & \vdots& \vdots\\
I_{k,1} & I_{k,2} & \dots &I_{k,k-1} & O_k & O_{k+1} & \dots & O_{n+1} \\
I_{k+1,1} & I_{k+1,2} & \dots &I_{k+1,k-1} & I_{k+1,k} & O_{k+1} & \dots & O_{n+1} \\
  \vdots& \vdots & \vdots &\vdots & \vdots& \vdots & \vdots& \vdots\\
  I_{n+1,1} & I_{n+1,2} & \dots &I_{n+1,k-1} & I_{n+1,k} & I_{n+1,k+1} & \dots & O_{n+1} \\
\end{array}
\right).
$$
In other words, the $i$-th row of $J_{k-1}$ is the same as for $A$ provided $k \le i \le n+1$ and
in the case $1 \le i \le k-1$ the $i$-th and the $k$-th rows of $J_{k-1}$ are the same. In particular, the ideal
$J = J_n$ has the shape
\begin{equation*}
J =
\left(
\begin{array}{llccc}
I_{n+1,1} & I_{n+1, 2} & \dots & I_n & O_{n+1} \\
I_{n+1,1} & I_{n+1, 2} & \dots & I_n & O_{n+1} \\
\vdots & \vdots & \vdots & \vdots & \vdots \\
I_{n+1,1} & I_{n+1, 2} & \dots & I_n & O_{n+1} \\
I_{n+1,1} & I_{n+1, 2} & \dots & I_n & O_{n+1} \\
\end{array}
\right).
\end{equation*}
Consider the projective left $A$--module $P := Ae$. Then we have an adjoint pair
$$
\xymatrix{A-\mathsf{mod} \ar@/^2ex/[rr]|{\,\tilde\sG\,} & & \widetilde{O}-\mathsf{mod} \ar@/^2ex/[ll]|{\,\tilde\sF\,}
}
$$
 where
$\tilde\sG = \Hom_A(P, \,-\,)$ and $\tilde\sF = P \otimes_{\widetilde{O}} \,-\,$. The functor $\tilde\sF$ is exact and has the following explicit description: if $M$ is an $\widetilde{O}$--module then
$$
\tilde\sF(M) = M^{\oplus{(n+1)}} =
\left(
\begin{array}{c}
M \\
M\\
\vdots \\
M
\end{array}
\right)
$$
where the left $A$--action on $M^{\oplus{(n+1)}}$ is given by the matrix multiplication. Since for every
$1 \le k \le n$ the $O$--module $I_{n+1, k}$ is also a projective  $\widetilde{O}$--module, we see that the left $A$--module $J e_k$
belongs to the essential image of $\tilde\sF$ and is projective over $A$. It is clear that all right $A$--modules
$e_k J$ are projective, too. Since $P$ is free over $\widetilde{O} = \End_{A}(P)$,  \cite[Lemma 4.9]{bdg} implies that
$
\mathrm{gl.dim}(A) \le \mathrm{gl.dim}(Q) + 2.
$

\noindent
Next, observe that for every $1 \le k \le n$ the ring $Q_k$ has the following ``matrix description'':
$$
Q_k  \cong
\left(
\begin{array}{llcc}
\dfrac{O_1}{I_{k+1,1}} & \dfrac{O_2}{I_{k+1, 2}} & \dots & \dfrac{{O}_k}{I_k}  \\
\dfrac{I_{2,1}}{I_{k+1,1}} & \dfrac{O_2}{I_{k+1, 2}} & \dots & \dfrac{{O}_k}{I_k}  \\
\vdots & \vdots & \ddots & \vdots  \\
\dfrac{I_{k,1}}{I_{k+1,1}} & \dfrac{I_{k,2}}{I_{k+1, 2}} & \dots &  \dfrac{{O}_k}{I_k}  \\
\end{array}
\right),
$$
where $\dfrac{{O}_k}{I_k} =: \bar{O}_k$ is semi--simple.
For $1 \le k \le n$ let  $\bar{e}_k$ be the image
of the idempotent $e_k \in A$ in the ring $Q_k = A/J_k$. Observe that for  $2 \le k \le n$
$$L_k := J_{k-1}/J_k = Q_k \bar{e}_k Q_k  =
Q_k  \cong
\left(
\begin{array}{llcc}
\dfrac{I_{k,1}}{I_{k+1,1}} & \dfrac{I_{k,2}}{I_{k+1, 2}} & \dots &  \dfrac{{O}_k}{I_k} \\
\dfrac{I_{k,1}}{I_{k+1,1}} & \dfrac{I_{k,2}}{I_{k+1, 2}} & \dots &  \dfrac{{O}_k}{I_k} \\
\vdots & \vdots & \ddots & \vdots  \\
\dfrac{I_{k,1}}{I_{k+1,1}} & \dfrac{I_{k,2}}{I_{k+1, 2}} & \dots &  \dfrac{{O}_k}{I_k}  \\
\end{array}
\right)
\subset Q_k$$
is projective viewed both as a left and as a right
$Q_k$--module (via the same argument as for $J$ and $A$). Moreover, $Q_k/L_k \cong Q_{k-1}$ and
$\bar{e}_k Q_k \bar{e}_k = \bar{O}_k$ is semi--simple.
Therefore, $J_1/J \subset J_2/J \subset \dots \subset J_n/J$ is a \emph{heredity chain} in $Q$ and the ring
$Q$ is \emph{quasi--hereditary}, see \cite{CPS,dr} or the appendix of Dlab in \cite{dk} for the definition and main properties of quasi--hereditary rings. It is well--known that
$
\mathrm{gl.dim}(Q) \le 2(n-1),
$
see \cite[Statement 9]{dr}, \cite[Theorem A.3.4]{dk} (or \cite[Lemma 4.9]{bdg} for a short proof). The theorem is proven.
\end{proof}

\begin{remark}
The bound on the global dimension of $A$ given in Theorem \ref{T:GlobalDimensionKoenigOrder} is not optimal.
For example, if $O = \kk\llbracket u, v\rrbracket/(u^2 - v^{m(n)})$ with $m(n) = 2n$ (respectively $2n+1$) is a simple singularity of type
$A_{m(n)-1}$, then the level of $O$ is $n$. On the other hand, $O_1 \oplus \dots \oplus O_{n+1}$ is the additive generator of the category of maximal Cohen--Macaulay modules, see \cite[Section 7]{Bass}, \cite[Section 5]{lw} or
\cite[Section 9]{Yoshino}. Hence, by a result of Auslander and Roggenkamp \cite{ar}, the global dimension of $A$ is equal to two.

\smallskip
\noindent
In the particular cases $O = \kk\llbracket u, v\rrbracket/(u^2 - v^{2})$ (simple node) and  $O = \kk\llbracket u, v\rrbracket/(u^2 - v^{3})$ (simple cusp) the K\"onig's order $A$ coincides with the Auslander's order
$
\left(
\begin{array}{cc}
O & \widetilde{O} \\
C & \widetilde{O}
\end{array}
\right)
$
introduced in the work \cite{bd}.
\end{remark}

\begin{remark}
Basic properties of excellent rings (see \cite[Section 6]{ddone} or \cite[Section 8.2]{liu}) imply that
$$Q_O := Q \cong Q_{\widehat{O}_1} \times \dots \times Q_{\widehat{O}_t},$$ where $\widehat{O}_i$ is the completion of the local ring
$O_{\idm_i}$ for each $\idm_i \in \mathrm{Sing}(O)$. In other words, the quasi--hereditary ring $Q$ depends only on
the \emph{local singularity types} of $\Spec(O)$.
\end{remark}

\section{K\"onig's order as a categorical resolution of singularities}
For a (left) Noetherian ring $B$ we denote by $B-\mathsf{mod}$ the category of all finitely generated left $B$--modules and by $B-\mathsf{Mod}$ the category of all left $B$--modules.
As in the previous section, let $O$ be an excellent reduced Noetherian ring of Krull dimension one and level $n$, $\widetilde{O}$ be its normalization, $A$ be the K\"onig's order of $O$ and $Q$ be the quasi--hereditary artinian  algebra attached to $O$. Let $e = e_{n+1}$ and $f = e_1$ be two standard idempotents of $A$, $P = Ae$, $T = Af$ and
$J = AeA$.
It is clear that $\widetilde{O} \cong \End_A(P)$ and $O \cong \End_A(T)$.
We also denote $T^\vee := \Hom_A(T, A) \cong f A$ and $P^\vee := \Hom_A(P, A) \cong e A$.
Then we have the following diagram of categories and functors:
\begin{equation}
\xymatrix{O-\mathsf{mod} \ar@/^2ex/[rr]^{\,\sF\,} \ar@/_2ex/[rr]_{\,\sH\,} && A-\mathsf{mod} \ar[ll]|{\,\sG\,}
  \ar[rr]|{\,\tilde\sG\,} && \widetilde{O}-\mathsf{mod} \ar@/^2ex/[ll]^{\,\tilde\sH\,} \ar@/_2ex/[ll]_{\,\tilde\sF\,}}
\end{equation}
where $\sF = T \otimes_O \,-\,$, $\sH = \Hom_O(T^\vee, \,-\,)$, $\sG = \Hom_A(T, \,-\,)$ and similarly,
$\tilde\sF = P \otimes_{\widetilde{O}} \,-\,$, $\tilde\sH = \Hom_{\widetilde{O}}(P^\vee, \,-\,)$,
$\tilde\sG = \Hom_A(P, \,-\,)$. There is the same diagram for the categories of all modules
$O-\mathsf{Mod}$, $\widetilde{O}-\mathsf{Mod}$ and $A-\mathsf{Mod}$. The following results are standard, see for example \cite[Theorem 4.3]{bdg} and references therein.

\begin{theorem}
The pairs of functors $(\sF, \sG)$, $(\sG, \sH)$ (and respectively $(\tilde\sF, \tilde\sG)$, $(\tilde\sG, \tilde\sH)$) are adjoint and the functors $\sF, \sH, \tilde\sF$ and $\tilde\sH$ are fully faithful. Both categories
$O-\mathsf{mod}$ and $\widetilde{O}-\mathsf{mod}$ are Serre quotients of $A-\mathsf{mod}$:
$$
O-\mathsf{mod} \cong A-\mathsf{mod}/\mathsf{Ker}(\sG) \quad \mbox{and} \quad
\widetilde{O}-\mathsf{mod} \cong A-\mathsf{mod}/\mathsf{Ker}(\tilde\sG).
$$
Moreover,
$
\mathsf{Ker}(\tilde\sG) = Q-\mathsf{mod}.
$
\end{theorem}
The described picture becomes even better when we pass to (unbounded) derived categories. Observe that
the functors $\sG, \tilde\sG, \tilde\sF$ and $\tilde\sH$ are exact. Their derived functors will be denoted
by $\sD\sG, \sD\tilde\sG, \sD\tilde\sF$ and $\sD\tilde\sH$ respectively, whereas $\sL\sF$ is the left derived functor of $\sF$ and $\sR\sH$ is the right derived functor of $\sH$.

\begin{theorem} We have a diagram of categories and functors
\begin{equation}\label{E:descinglobale}
\xymatrix{D(O-\mathsf{Mod}) \ar@/^2ex/[rr]^{\,\sL\sF\,} \ar@/_2ex/[rr]_{\,\sR\sH\,} && D(A-\mathsf{Mod}) \ar[ll]|{\,\sD\sG\,}
  \ar[rr]|{\,\sD\tilde\sG\,} && D(\widetilde{O}-\mathsf{Mod}) \ar@/^2ex/[ll]^{\,\sD\tilde\sH\,} \ar@/_2ex/[ll]_{\,\sD\tilde\sF\,}}
\end{equation}
satisfying the following properties.
\begin{itemize}
\item The following pairs  of functors $(\sL\sF, \sD\sG)$, $(\sD\sG, \sR\sH)$, $(\sD\tilde\sF, \sD\tilde\sG)$ and
$(\sD\tilde\sG, \sD\tilde\sH)$ form adjoint pairs.
\item The functors $\sL\sF$, $\sR\sH$, $\sD\tilde\sF$ and $\sD\tilde\sH$ are fully faithful.
\item Both derived categories $D(O-\mathsf{Mod})$ and $D(\widetilde{O}-\mathsf{Mod})$ are Verdier localizations of
$D(A-\mathsf{Mod})$:
\begin{itemize}
  \item $
D(O-\mathsf{Mod}) \cong D(A-\mathsf{Mod})/\mathsf{Ker}(\sD\sG)$.
\item
$D(\widetilde{O}-\mathsf{Mod}) \cong D(A-\mathsf{Mod})/\mathsf{Ker}(\sD\tilde\sG).
$
\end{itemize}
\item Moreover,
$
\mathsf{Ker}(\sD\tilde\sG) = D_Q(A-\mathsf{Mod}) \cong D(Q-\mathsf{Mod}).
$
\item
The derived category $D(A-\mathsf{Mod})$ is a categorical resolution of singularities of $X = \Spec(O)$ in the sense of Kuznetsov  \cite[Definition 3.2]{ku}.
 \item If $O$ is Gorenstein, then the restrictions of $\sL\sF$ and $\sR\sH$ on $\Perf(O)$ are isomorphic. Hence, the constructed categorical resolution is even \emph{weakly crepant} in the sense of
     \cite[Definition 3.4]{ku}.
\end{itemize}
We have a \emph{recollement diagram}
\begin{equation}\label{E:recollement}
\xymatrix{D(Q-\mathsf{Mod}) \ar[rr]|{\,\sI\,} && D(A-\mathsf{Mod}) \ar@/^2ex/[ll]^{\,\sI^*\,} \ar@/_2ex/[ll]_{\,\sI^{!}\,} \ar[rr]|{\,\sD\tilde\sG\,}
  && D(\widetilde{O}-\mathsf{Mod}) \ar@/^2ex/[ll]^{\,\sD\tilde\sH\,} \ar@/_2ex/[ll]_{\,\sD\tilde\sF\,}}
\end{equation}
and all functors can be restricted on the bounded derived categories $D^b(Q-\mathsf{mod})$, $D^b(A-\mathsf{mod})$
and $D^b(\widetilde{O}-\mathsf{mod})$. In particular, we have \emph{two} semi--orthogonal decompositions
$$
D(A-\mathsf{Mod}) = \bigl\langle \mathsf{Ker}(\sD\tilde\sG), \mathsf{Im}(\sL\sF)\bigr\rangle =
\bigl\langle \mathsf{Im}(\sR\sH), \mathsf{Ker}(\sD\tilde\sG)\bigr\rangle.
$$
The same result is true when we pass to the bounded derived categories.
\end{theorem}

\noindent
\emph{Comment on the proof}. The study of various derived functors  related with  a pair $(B, \epsilon)$, where $B$ is a ring and $\epsilon \in B$ an idempotent  (in particular, the recollement diagram (\ref{E:recollement})) are due to Cline, Parshall and Scott \cite[Section 2]{CPS}. We also refer to \cite[Section 4]{bdg} (and references therein) for an  exposition focussed on  non--commutative resolutions of singularities.
The weak crepancy of the categorical resolution $D(A-\mathsf{Mod})$ of $\Spec(O)$ follows from \cite[Theorem 5.10]{bdg}. In particular, the constructed categorical  resolution of singularities
fits into the setting of non--commutative crepant resolutions initiated by van den Bergh in \cite{vdb}. \qed

\section{Survey on the derived stratification of an artinian  quasi--hereditary ring}
The derived category $D^b(Q-\mathsf{mod})$ of the quasi--hereditary ring $Q$ introduced in Theorem \ref{T:GlobalDimensionKoenigOrder} can be further stratified in a usual way \cite{CPS}, which we briefly describe now adapting the notation for further applications. All details can be found in \cite{CPS}, \cite[Appendix]{dr} and \cite{bdg}.

\medskip
\noindent
1.~Recall that we had started with a reduced excellent Noetherian ring $O$ of Krull dimension one, attaching to it
a certain order $A$. Then we have constructed a heredity chain $J_n \subset J_{n-1} \subset \dots \subset
J_1 \subset A$ of two--sided ideals and posed $Q_k := A/J_k$ for $1 \le k \le n$. In this notation,
$Q = Q_n$ is an artinian quasi--hereditary ring we shall study in this section and $Q_1 = \bar{O}$ is a semi-simple ring (supported on the singular locus of $\Spec(O)$).

\medskip
\noindent
2.~For any $1 \le k \le n$, let
$\bar{e}_k$ be the image of the standard idempotent $e_k \in A$ in $Q_k = A/J_k$. Then $Q_k/(Q_k \bar{e}_k Q_k) \cong Q_{k-1}$ for all $2 \le k \le n$.

\smallskip
\noindent
Let
$P_k = Q_k \bar{e}_k$ be the  projective left $Q_k$--module and
$P_k^\vee = \Hom_{Q_k}(P_k, Q_k) = e_k Q_k$ be the  projective right $Q_k$--module, corresponding to
the idempotent $\bar{e}_k$. Then we have:  $\End_{Q_k}(P_k) \cong \bar{O}_k$.
The functor
$$
\sG_k = \Hom_{Q_k}(P_k,\,-\,): Q_k-\mathsf{mod} \lar \bar{O}_k-\mathsf{mod}
$$
is a bilocalization functor: the functors $\sF_k = P_k \otimes_{\bar{O}_k} \,-\,$ and  $\sH_k = \Hom_{Q_k}(P_k^\vee, \,-\,)$ are respectively the left and the right adjoints of $\sG_k$. Both $\sF_k$ and $\sH_k$ are fully faithful. Since
the ring $\bar{O}_k$ is semi--simple,  $\sF_k$ and $\sH_k$ are also exact. The kernel of $\sG_k$ is the category of $Q_{k-1}$--modules.

\medskip
\noindent
3.~Most remarkably, for any $2 \le k \le n$ we have a recollement diagram
\begin{equation*}
\xymatrix{D^b(Q_{k-1}-\mathsf{mod}) \ar[rr]|-{\,\sJ_{k}\,} && D^b(Q_k-\mathsf{mod}) \ar@/^2ex/[ll]^-{\,\sJ_k^*\,} \ar@/_2ex/[ll]_-{\,\sJ_k^{!}\,} \ar[rr]|-{\,\sD\sG_k\,}
  && D^b(\bar{O}_k-\mathsf{mod}) \ar@/^2ex/[ll]^-{\,\sD\sH_k\,} \ar@/_2ex/[ll]_-{\,\sD\sF_k\,}}
\end{equation*}
This  claim   in particular includes the following statements.
\begin{itemize}
\item The functor $\sJ_k$ (induced by the ring homomorphism $Q_k \lar Q_{k-1}$) is fully faithful.
The essential image of $\sJ_k$ coincides with the kernel of $\sD\sG_k$ and
$D^b(Q_k-\mathsf{mod})/\mathsf{Im}(\sJ_k) \cong D^b(\bar{O}_k-\mathsf{mod})$.
\item The functors $\sD\sF_k$ and $\sD\sH_k$ are fully faithful.
\end{itemize}

\medskip
\noindent
4.~For all $1 \le k \le n$ we have:
\begin{itemize}
\item $\sD\sF_k(\bar{O}_k) \cong \sF_k(\bar{O}_k) \cong P_k$.
\item $\sD\sH_k(\bar{O}_k) \cong \sH_k(\bar{O}_k) = \Hom_{\bar{O}_k}(P_k^\vee, \bar{O}_k) := E_k$ is the injective left $Q_k$--module corresponding to the idempotent $\bar{e}_k$.
\end{itemize}
The functor $\sI_k: D^b(Q_k-\mathsf{mod}) \lar D^b(Q-\mathsf{mod})$ induced by the ring epimorphism
$Q \longrightarrow Q_k$ is fully faithful. In fact, it admits a factorization $\sI_k = \sJ_n   \dots \sJ_{k+1}$.
The $Q$--module $\Delta_k := \sI_k(P_k)$ (respectively $\nabla_k := \sI_k(E_k)$)  is called  $k$-th \emph{standard} (respectively \emph{costandard}) $Q$--module.

\medskip
\noindent
5.~The standard and costandard modules have in particular the following properties:
$$
\Ext_Q^p(\Delta_i, \Delta_j) = 0 = \Ext_Q(\nabla_j, \nabla_i) \; \mbox{for all} \; 1 \le i < j \le n \; \mbox{and} \; p \ge 0
$$
and
$$
\Ext_Q^p(\Delta_k, \Delta_k) = 0 = \Ext_Q^p(\nabla_k, \nabla_k) \; \mbox{for all} \; 1 \le k \le n \; \mbox{and} \; p \ge 1.
$$
Moreover, $\End_Q(\Delta_k) \cong \End_Q(\nabla_k) \cong \bar{O}_k$ is semi--simple. The derived category
$D^b(Q-\mathsf{mod})$ admits two canonical semi--orthogonal decompositions:
$$
\langle D_1, \dots, D_n\rangle = D^b(Q-\mathsf{mod}) = \langle D'_n, \dots, D'_1\rangle,
$$
where $D_k$ (respectively $D'_k$) is the triangulated subcategory of $D^b(Q-\mathsf{mod})$ generated by the object
$\Delta_k$ (respectively $\nabla_k)$. Note that we have the following equivalences of categories:
$
D_k \cong D^b(\bar{O}_k-\mathsf{mod}) \cong D_k'.
$

\medskip
\noindent
6.~The stratification of $D^b(Q-\mathsf{mod})$ by the derived categories $D^b(\bar{O}_k-\mathsf{mod})$ can be summarized by the following diagram of categories and functors:
$$
\xymatrix{
D^b(\bar{O}_1-\mathsf{mod}) \ar[d]_= &  D^b(\bar{O}_2-\mathsf{mod}) \ar[d]_{\sD\sF_2} & & D^b(\bar{O}_n-\mathsf{mod}) \ar[d]^{\sD\sF_n} \\
D^b(Q_1-\mathsf{mod}) \ar[r]^-{\sJ_2} &  D^b(Q_2-\mathsf{mod}) \ar[r]^-{\sJ_3} &  \dots \ar[r]^-{\sJ_{n}} & D^b(Q_n-\mathsf{mod})\\
D^b(\bar{O}_1-\mathsf{mod}) \ar[u]^=  & D^b(\bar{O}_2-\mathsf{mod}) \ar[u]^{\sD\sH_2} & & D^b(\bar{O}_n-\mathsf{mod}) \ar[u]_{\sD\sH_n}
}
$$
\section{Derived stratification  and curve singularities}
Recall that we also have the following recollement diagram
\begin{equation*}
\xymatrix{D^b(Q-\mathsf{mod}) \ar[rr]|{\,\sI\,} && D^b(A-\mathsf{mod}) \ar@/^2ex/[ll]^{\,\sI^*\,} \ar@/_2ex/[ll]_{\,\sI^{!}\,} \ar[rr]|{\,\sD\tilde\sG\,}
  && D^b(\widetilde{O}-\mathsf{mod}) \ar@/^2ex/[ll]^{\,\sD\tilde\sH\,} \ar@/_2ex/[ll]_{\,\sD\tilde\sF\,}}
\end{equation*}
where $\sI$ is induced by the ring epimorphism $A \rightarrow Q$. Abusing the notation, we shall write
$\Delta_k = \sI(\Delta_k)$ for all $1 \le k \le n$. This implies the following result.

\begin{theorem}
The derived category $D^b(A-\mathsf{mod})$ admits two semi--orthogonal decompositions
$
\bigl\langle\mathsf{Im}(\sI), \mathsf{Im}(\sL\tilde\sF)\bigl\rangle = D^b(A-\mathsf{mod}) =
\bigl\langle\mathsf{Im}(\sR\tilde\sH), \mathsf{Im}(\sI)\bigl\rangle.
$
\end{theorem}

\noindent
Next, recall that we have a bilocalization functor $$\sD\sG: D^b(A-\mathsf{mod}) \lar D^b(O-\mathsf{mod}).$$
\begin{lemma}\label{L:calcullocale}
For any $1 \le k \le n$ we have: $\sD\sG(\Delta_k) \cong \bar{O}_k$. Moreover, $\sD\sG(Q) \cong
O_1/C_1 \oplus \dots \oplus O_n/C_n$, where $C_k := I_{n+1,k} = \Hom_O\bigl(O_{n+1}, O_k\bigr)$.
\end{lemma}

\begin{proof} The first result follows from the following chain of isomorphisms:
$$
\sD\sG(\Delta_k) \cong \sG(\Delta_k) = \Hom_A(Af, \Delta_k) \cong  f\cdot \Delta_k \cong \bar{O}_k.
$$
The proof of the second statement  is analogous.
\end{proof}

\section{K\"onig's resolution in the projective setting}\label{S:ProjectiveSetting}

\smallskip
\noindent
Let $X$ be a reduced projective curve over some base field $\kk$. In this section we shall explain the construction
of K\"onig's sheaf of orders $\kA$, ``globalizing'' the arguments of Section \ref{Sec:LocalStory}.
\begin{itemize}
\item Let $\widetilde{X} \stackrel{\nu}\lar X$ be the normalization of $X$ and $Z$ be the singular locus of $X$
(equipped with the reduced scheme structure).
\item In what follows, $\kO = \kO_X$ is the structure sheaf of $X$, $\kK$ is the  sheaf of rational functions on $X$,  $\widetilde{\kO}: = \nu_*(\kO_{\widetilde{X}})$
and $\kI$ is the ideal sheaf of the singular locus $Z$.
\item We consider the sheaf of rings $\kO^\sharp := \kEnd_X(\kI)$ on the curve $X$.
\end{itemize}
The next result   follows from the corresponding affine version (Proposition \ref{P:DetailsKoenigResol}).
\begin{proposition}
We have inclusions of sheaves $\kO \subseteq \kO^\sharp \subseteq \widetilde{O}$, and $\kO =  \kO^\sharp$ if and only if $X$ is smooth. Moreover, there are isomorphisms of $\kO$--modules
$
\kI \cong \kHom_X(\widetilde{\kO}, \kO)$ and  $\widetilde{\kO} \cong \kHom_X(\kI, \kO)$.
\end{proposition}

\noindent
Now we define a sequence of sheaves of rings $\kO \subset \kO_k \subset \widetilde{\kO}$ by the following recursive procedure.
\begin{itemize}
\item First we pose: $\kO_1 := \kO$.
\item Assume that the sheaf of rings $\kO_k$  has been  constructed. Then it defines a projective curve $X_k$ together with  a finite birational morphism $\nu_k: X_k \lar X$ (partial normalization of $X$) such that $\kO_k = \bigl(\nu_k\bigr)_*(\kO_{X_k})$.
\item Let $Z_k$ be the singular locus of the curve $X_k$ (as usual, with respect to the reduced scheme structure).
Then we write
$$
\kO_{k+1} := \kO_k^\sharp \cong  (\nu_k)_*\bigl(\kEnd_{X_k}(\kI_{Z_k})\bigr).
$$
\end{itemize}
Then there exists a natural number $n$ (called the \emph{level} of $X$) such that we have a finite chain of sheaves of rings
$$
\kO = \kO_1 \subset \kO_2 \subset \dots \subset \kO_n \subset \kO_{n+1} = \widetilde{\kO}.
$$
Obviously, the level of $X$ is the maximum of the levels of local rings $\widehat{\kO}_x$,
where $x$ runs through the set of singular points of $X$.

\begin{definition}
The sheaf of rings
$\kA := \kEnd_X\bigl(\kO_1 \oplus \dots \oplus \kO_{n+1}\bigr)$ is called the \emph{K\"onig's sheaf of orders} on $X$.
\end{definition}

\noindent
In what follows, we study the ringed space $\bbX = (X, \kA)$. We denote by $\Coh(\bbX)$ (respectively
$\Qcoh(\bbX)$) the category of coherent (respectively quasi--coherent) sheaves of $\kA$--modules on the curve $X$.

\begin{theorem} The sheaf of orders $\kA$ admits the following description:
\begin{equation}\label{E:KoenigOrderSheaf}
\kA \cong
\left(
\begin{array}{llccc}
\kO_1 & \kO_2 & \dots & \kO_n & \kO_{n+1} \\
\kI_1 & \kO_2 & \dots & \kO_n & \kO_{n+1} \\
\vdots & \vdots & \ddots & \vdots & \vdots \\
\kI_{n,1} & \kI_{n, 2} & \dots & \kO_n & \kO_{n+1} \\
\kI_{n+1,1} & \kI_{n+1, 2} & \dots & \kI_n & \kO_{n+1} \\
\end{array}
\right) \subset \kMat_{n+1, n+1}(\kK),
\end{equation}
where $\kI_{i, j} := \kHom_{X}(\kO_i, \kO_j)$ for all $1 \le j < i\le n+1$ and $\kI_k = \kI_{k+1, k}$ for $1 \le k \le n$. Moreover, $\kA \otimes_\kO \kK \cong \kMat_{n+1, n+1}(\kK)$. Next, we have:
$$
\mathrm{gl.dim}\bigl(\Coh(\bbX)\bigr) = \mathrm{gl.dim}\bigl(\Qcoh(\bbX)\bigr) \le 2n,
$$
where $n$ is the level of $X$.
\end{theorem}

\begin{proof}
The result follows from the corresponding local statements in Proposition \ref{P:KoenigOrderMatrDescr} and
Theorem \ref{T:GlobalDimensionKoenigOrder} and the fact that
$$
\mathrm{gl.dim}\bigl(\Coh(\bbX)\bigr) = \mathrm{gl.dim}\bigl(\Qcoh(\bbX)\bigr) = \max\bigl\{\mathrm{gl.dim}(\widehat{\kA}_x) \, | \, x \in X_{\mathrm{cl}}\bigr\},
$$
see for instance \cite[Corollary 5.5]{bdg}.
\end{proof}

For any $1 \le i \le n+1$, let $e_i \in \Gamma(X, \kA)$ be the $i$-th standard idempotent with respect to the matrix
presentation (\ref{E:KoenigOrderSheaf}). As in the affine case, we use the following notation.
\begin{itemize}
\item We write $e = e_{n+1}$ and $f = e_1$. Let $\kP := \kA e$ and $\kT := \kA f$ be the corresponding locally projective left $\kA$--modules. Then we have the following isomorphisms of sheaves of $\kO$--algebras:
\begin{equation}\label{E:keysheafisom}
\kO \cong \kEnd_{\bbX}(\kT) := \kEnd_{\kA}(\kT) \ \mbox{ and } \ \widetilde{\kO} \cong \kEnd_{\bbX}(\kP) :=
 \kEnd_{\kA}(\kP).
\end{equation}
We shall also use the notation
$$
\kP^\vee := \kHom_{\bbX}(\kP, \kA) \cong e\kA \ \mbox{ and } \ \kT^\vee := \kHom_{\bbX}(\kT, \kA) \cong f\kA.
$$
\item For any $1 \le k \le n$ we set
$$\varepsilon_k := \sum\limits_{i = k+1}^{n+1} e_i \in \Gamma(X, \kA).
$$
Then $\kJ_k:= \kA \varepsilon_k\kA$ denotes  the corresponding sheaf of two--sided ideals in $\kA$.
\item The sheaves of $\kO$--algebras $\kQ_k := \kA/\kJ_k$ are supported on the finite set $Z$ for all $1 \le k \le n$. In what follows, we shall identify them with the corresponding finite dimensional
    $\kk$--algebras of global sections
 $
 Q_k := \Gamma(X, \kQ_k),
 $
which have been shown to be quasi--hereditary, see Theorem \ref{T:GlobalDimensionKoenigOrder}. As before, we shall write  $\kJ = \kJ_n$ and $Q = Q_n$.
\item In a similar way, the torsion sheaf  $\kO_k/\kI_k$ will be identified with the corresponding ring of global sections
$
\bar{O}_k:= \Gamma\bigl(X, \kO_k/\kI_k),
$
which is a  semi--simple finite dimensional $\kk$--algebra, isomorphic  to  the ring of functions of the singular locus $Z_k$ of the partial normalization $X_k$ of our original curve $X$.
\end{itemize}

\begin{proposition}\label{P:recollementSheaves} Consider the following diagram of categories and functors
\begin{equation}
\xymatrix{\Coh(X) \ar@/^2ex/[rr]^{\,\sF\,} \ar@/_2ex/[rr]_{\,\sH\,} && \Coh(\bbX) \ar[ll]|{\,\sG\,}
  \ar[rr]|{\,\tilde\sG\,} && \Coh(\widetilde{X}) \ar@/^2ex/[ll]^{\,\tilde\sH\,} \ar@/_2ex/[ll]_{\,\tilde\sF\,}}
\end{equation}
where $\sF = \kT \otimes_O \,-\,$, $\sH = \kHom_X(\kT^\vee, \,-\,)$, $\sG = \kHom_{\bbX}(\kT, \,-\,)$ and similarly,
$\tilde\sF = \kP \otimes_{\widetilde{O}} \,-\,$, $\tilde\sH = \kHom_{\widetilde{X}}(P^\vee, \,-\,)$,
$\tilde\sG = \kHom_{\bbX}(\kP, \,-\,)$. Here we identify (using the functor $\nu_*$) the category $\Coh(\widetilde{X})$ with the category of coherent $\widetilde{O}$--modules on the curve $X$. Then the following results are true.
\begin{itemize}
\item The  pairs of functors $(\sF, \sG), (\sG, \sH)$ and $(\tilde\sF, \tilde\sG), (\tilde\sG, \tilde\sH)$ form adjoint pairs. The functors $\sF, \sH, \tilde\sF$ and $\tilde\sH$ are fully faithful.
\item The functors $\sG$ and $\tilde\sG$ are bilocalization functors. Moreover, $\mathsf{Ker}(\sF) \cong
Q-\mathsf{mod}$.
\item The pairs of functors $(\tilde\sG\sF, \sG\tilde\sH)$ and $(\sG\tilde\sF, \tilde\sG\sH)$ between
$\Coh(X)$ and $\Coh(\widetilde{X})$ form adjoint pairs, too. Moreover, these functors  admit the following ``purely commutative'' descriptions:
$$
\sG\tilde\sF \simeq \nu_*,\ \tilde\sG \sH \simeq \nu^!,\ \tilde\sG \sF \simeq \kC \otimes_{\widetilde\kO} \nu^*(\,-\,)
\ \mbox{\emph{ and }} \ \sG\tilde\sH  \simeq  \nu_*(\kC^\vee \otimes_{\widetilde{\kO}} \,-\,),
$$
where $\kC := \kHom_X(\widetilde{\kO}, \kO) = \kI_{n+1,1}$ is the conductor ideal sheaf.
\end{itemize}
The same results are true when we replace each category of coherent sheaves by the corresponding category of quasi--coherent sheaves.
\end{proposition}

\begin{proof}
The proofs of the first two parts follow from standard local computations. Let $\kF$ be a coherent $\kO$--module and
$\kG$ a coherent $\widetilde{O}$--module (identified with the corresponding coherent sheaf on $\widetilde{X}$).
Then we have:
$$
\tilde\sG\sF(\kF) = \kHom_{\bbX}(\kA e, \kA f \otimes_{\kO} \kF) \cong \bigl(e \kA f\bigr)\otimes_{\kO} \kF \cong
\kC \otimes_{\kO} \kF \cong \kC \otimes_{\kO} \bigl(\widetilde\kO \otimes_{\kO} \kF\bigr)
$$
and
$$
\sG\tilde\sF(\kG) = \kHom_{\bbX}(\kA f, \kA e \otimes_{\widetilde\kO} \kG) \cong
\bigl(f \kA e\bigr)\otimes_{\widetilde\kO} \kG \cong \widetilde\kO \otimes_{\widetilde\kO} \kG \cong \kG.
$$
This  proves the isomorphisms of functors $\tilde\sG \sF \simeq \kC \otimes_{\widetilde\kO} \nu^*(\,-\,)$ and
$\sG\tilde\sF \simeq \nu_*$. Since $\sG\tilde\sH$ and $\tilde\sG\sH$ are right adjoints of $\tilde\sG \sF$ and $\sG\tilde\sF$
 respectively, the remaining isomorphisms are true as well.
\end{proof}

\noindent
The next statement  summarizes the main properties of the K\"onig's resolution in the projective framework.
\begin{theorem}\label{T:main} We have a diagram of categories and functors
\begin{equation}\label{E:descinglobale2}
\xymatrix{D\bigl(\Qcoh(X)\bigr) \ar@/^2ex/[rr]^{\,\sL\sF\,} \ar@/_2ex/[rr]_{\,\sR\sH\,} && D(\Qcoh(\bbX) \ar[ll]|{\,\sD\sG\,}
  \ar[rr]|{\,\sD\tilde\sG\,} && D\bigl(\Qcoh(\widetilde{X})\bigr) \ar@/^2ex/[ll]^{\,\sD\tilde\sH\,} \ar@/_2ex/[ll]_{\,\sD\tilde\sF\,}}
\end{equation}
satisfying the following properties.
\begin{itemize}
\item The pairs  of functors $(\sL\sF, \sD\sG)$, $(\sD\sG, \sR\sH)$, $(\sD\tilde\sF, \sD\tilde\sG)$ and
$(\sD\tilde\sG, \sD\tilde\sH)$ form adjoint pairs.
\item The functors $\sL\sF$, $\sR\sH$, $\sD\tilde\sF$ and $\sD\tilde\sH$ are fully faithful.
\item Both derived categories $D\bigl(\Qcoh(X)\bigr)$ and $D\bigl(\Qcoh(\widetilde{X})\bigr)$ are Verdier localizations of
$D\bigl(\Qcoh(\bbX)\bigr)$:
\begin{itemize}
  \item $
D\bigl(\Qcoh(X)\bigr) \cong D\bigl(\Qcoh(\bbX)\bigr)/\mathsf{Ker}(\sD\sG)$.
\item
$D\bigl(\Qcoh(\widetilde{X})\bigr) \cong D\bigl(\Qcoh(\bbX)\bigr)/\mathsf{Ker}(\sD\tilde\sG).
$
\end{itemize}
\item Moreover,
$
\mathsf{Ker}(\sD\tilde\sG)  \cong D(Q-\mathsf{Mod}).
$
\item
The derived category $D\bigl(\Qcoh(\bbX)\bigr)$ is a categorical resolution of singularities of $X$ in the sense of Kuznetsov  \cite[Definition 3.2]{ku}.
 \item If $X$ is Gorenstein, then the restrictions of $\sL\sF$ and $\sR\sH$ on $\Perf(X)$ are isomorphic. Hence, the constructed categorical resolution is even \emph{weakly crepant} in the sense of
     \cite[Definition 3.4]{ku}.
\end{itemize}
We have a \emph{recollement diagram}
\begin{equation}\label{E:recollement2}
\xymatrix{D(Q-\mathsf{Mod}) \ar[rr]|{\,\sI\,} && D\bigl(\Qcoh(\bbX)\bigr) \ar@/^2ex/[ll]^{\,\sI^*\,} \ar@/_2ex/[ll]_{\,\sI^{!}\,} \ar[rr]|{\,\sD\tilde\sG\,}
  && D\bigl(\Qcoh(\widetilde{X})\bigr) \ar@/^2ex/[ll]^{\,\sD\tilde\sH\,} \ar@/_2ex/[ll]_{\,\sD\tilde\sF\,}}
\end{equation}
and all functors can be restricted on the bounded derived categories $D^b(Q-\mathsf{mod})$, $D^b\bigl(\Coh(\bbX)\bigr)$
and $D^b\bigl(\Coh(\widetilde{X})\bigr)$. In particular, we have \emph{two} semi--orthogonal decompositions
$$
D\bigl(\Qcoh(\bbX)\bigr) = \bigl\langle \mathsf{Ker}(\sD\tilde\sG), \mathsf{Im}(\sL\sF)\bigr\rangle =
\bigl\langle \mathsf{Im}(\sR\sH), \mathsf{Ker}(\sD\tilde\sG)\bigr\rangle.
$$
The same result is true when we pass to the bounded derived categories.
\end{theorem}

\begin{corollary}
For each $1 \le k \le n$ let $D_k$ (respectively $D'_k$) be the full subcategory of $D^b\bigl(\Coh(\bbX)\bigr)$ generated
by the $k$--th standard module $\Delta_k$ (respectively, the $k$--th costandard module $\nabla_k$). Then we have
equivalences of categories
$
D_k \cong D^b(\bar{O}_k-\mathsf{mod}) \cong D'_k
$
and semi--orthogonal decompositions
\begin{equation}\label{E:semiorthdec}
\bigl\langle D_1, \dots, D_n, \mathsf{Im}(\sL\tilde\sF)\bigr\rangle =
D^b\bigl(\Coh(\bbX)\bigr) = \bigl\langle \mathsf{Im}(\sR\tilde\sH), D'_n, \dots, D'_1\bigr\rangle.
\end{equation}
Both triangulated categories $\mathsf{Im}(\sL\tilde\sF)$ and $\mathsf{Im}(\sR\tilde\sH)$ are equivalent
to the derived category $D^b\bigl(\Coh(\widetilde{X})\bigr)$. Note that they are \emph{different} viewed as subcategories
of $D^b\bigl(\Coh(\bbX)\bigr)$.
\end{corollary}

\begin{remark}
As in the setting at the beginning of this section, let $X$ be a reduced excellent  curve, $\widetilde{X} \stackrel{\nu}\lar  X$ its normalization and
$\kC := \mathit{Hom}_X(\widetilde{\kO}, \kO)$ the conductor ideal. Then $\kC$ is also a sheaf of ideals in $\widetilde\kO$, hence the scheme $S = V(\kC) \stackrel{\eta}\hookrightarrow X$ is a \emph{non--rational locus of $X$ with respect to $\nu$} in the sense of Kuznetsov and Lunts \cite[Definition 6.1]{kl}. Starting with  the Cartesian diagram
$$
\xymatrix{
\widetilde{S} \ar[r]^{\tilde\eta} \ar[d]_{\tilde{\nu}} & \widetilde{X} \ar[d]^{\nu} \\
S \ar[r]^\eta  & X
}
$$
one can construct a partial categorical resolution of singularities of $X$ obtained by the ``naive gluing'' of the
derived categories $D\bigl(\Qcoh(\widetilde{X})\bigr)$ and $D\bigl(\Qcoh(S)\bigr)$, see \cite[Section 6.1]{kl}.
It would be interesting to compare the obtained triangulated category  with the derived category $D\bigl(\Qcoh(\widehat\bbX)\bigr)$ of the non--commutative curve
$
\widehat\bbX = \left(
X,
\left(
\begin{array}{cc}
\kO & \widetilde\kO \\
\kC & \widetilde\kO
\end{array}
\right)
\right),
$
see also \cite[Section 8]{bd}. Next, \cite[Theorem 6.8]{kl} provides a recipe to construct a categorical resolution of singularities of $X$, which however, involves some non--canonical choices.
It is an interesting question to compare these categorical resolutions  with K\"onig's resolution $\bbX$ constructed in our article. Another important problem is to give an ``intrinsic description'' of the derived category $D^b\bigl(\Coh(\bbX)\bigr)$, i.e.~to provide a list of properties describing it uniquely up to a triangle  equivalence.
We follow here the analogy with non--commutative crepant resolutions, see \cite[Conjecture 5.1]{bo} and  \cite[Conjecture 4.6]{vdb}. All such resolutions are known to be derived equivalent in certain cases, see for example \cite[Theorem 6.6.3]{vdb}. Recall that K\"onig's resolution $\bbX$ is weakly crepant in the case the curve $X$ is Gorenstein.
\end{remark}

\section{Purely commutative applications}
Results of  the previous sections allow to deduce a number of interesting  ``purely commutative'' statements.
Let $X$ be a reduced projective curve over some base field $\kk$ and
$\widetilde{X} \stackrel{\nu}\lar X$ be its normalization. According to Orlov \cite{or}, the Rouquier dimension of the derived category $D^b\bigl(\Coh(\widetilde{X})\bigr)$ is equal to one. In fact, Orlov constructs an explicit vector bundle
$\widetilde\kF$ on $\widetilde{X}$ such that $D^b\bigl(\Coh(\widetilde{X})\bigr) = \langle \widetilde\kF\rangle_2$ (here we follow the notation of   Rouquier's seminal article \cite{rou}).

\begin{theorem} Let $\kF := \nu_*(\widetilde{\kF})$ be the direct image of the Orlov's generator of $D^b\bigl(\Coh(\widetilde{X})\bigr)$. Then the following results are true.
\begin{itemize}
\item Let $Z$ be the singular locus of $X$ (with respect to the reduced scheme structure) and $\kO_Z$ be the corresponding structure sheaf. Then
\begin{equation}\label{E:1stEstderDim}
D^b\bigl(\Coh(X)\bigr) = \bigl\langle \kF \oplus \kO_Z\bigr\rangle_{n+2},
\end{equation}
where $n$ is the level of $X$.
\item Let $\kS = \kO_1/\kC_1 \oplus \dots \oplus \kO_n/\kC_n$, where $\kC_k := \kHom_X(\kO_{n+1}, \kO_k)$ is the conductor ideal sheaf of the $k$--th partial normalization of $X$ for $1 \le k \le n$.  Then we have:
\begin{equation}\label{E:2ndEstderDim}
D^b\bigl(\Coh(X)\bigr) = \bigl\langle \kF \oplus \kS\bigr\rangle_{d+3},
\end{equation}
where  $d$ is  the global dimension of the quasi--hereditary algebra $Q$ associated with  $X$.
\end{itemize}
\end{theorem}

\begin{proof}
According to Theorem \ref{T:main}, the derived category $D^b\bigl(\Coh(\bbX)\bigr)$ admits a semi--orthogonal decomposition
$$
D^b\bigl(\Coh(\bbX)\bigr) = \bigl\langle D^b(Q-\mathsf{mod}), \mathsf{Im}(\sL\tilde\sF)\bigr\rangle.
$$
Moreover, the derived category $D^b\bigl(\Coh(X)\bigr)$ is the Verdier localization of $D^b\bigl(\Coh(\bbX)\bigr)$ via the functor $\sD\sG$. This implies that whenever we have an object $\kX$ of $D^b\bigl(\Coh(\bbX)\bigr)$ with
$D^b\bigl(\Coh(\bbX)\bigr) = \langle \kX\rangle_m$ then $D^b\bigl(\Coh(X)\bigr) = \langle \sD\sG(\kX)\rangle_m$.
According to Proposition \ref{P:recollementSheaves} we have:
$$
(\sD\sG\cdot\sL\tilde\sF)(\widetilde\kF) \cong \sG\tilde\sF(\widetilde\kF) \cong \nu_*(\widetilde\kF) =: \kF.
$$
Next, Lemma \ref{L:calcullocale} implies that for all $1 \le k \le n$ we have:
$$
\sD\sG(\Delta_k) \cong \sG(\Delta_k) \cong \kO_k/\kI_k.
$$
Let $\nu_k: X_k \lar X$ be the $k$--th partial normalization of $X$ and $Z_k = \left\{y_1, \dots, y_p\right\}$ be the singular locus of $X_k$ (as usual, equipped with the reduced scheme structure). Then
$$\kO_k/\kI_k \cong (\nu_{k})_*\bigl(\kO_{X_k}/\kI_{Z_k}\bigr) \cong
(\nu_{k})_*\bigl(\kO_{Z_k}/\kI_{y_1} \oplus \dots \oplus  \kO_{Z_k}/\kI_{y_p}\bigr).
$$
Observe that if $y \in Z_k$ and $x = \nu_k(y)$ then $(\nu_{k})_*(\kO_{X_k}/\kI_y) \cong (\kO/\kI_x)^{\oplus l}$,
where $l = \mathsf{deg}\bigl[\kk_y: \kk_x\bigr]$.  Therefore,
$$
\mathsf{add}\bigl(\sG(\Delta_1) \oplus \dots \oplus \sG(\Delta_n)\bigr) = \mathsf{add}(\kO_Z)
$$
and (\ref{E:1stEstderDim}) is just a consequence of \cite[Lemma 3.5]{rou}. The  equality (\ref{E:2ndEstderDim}) follows in a similar way from Lemma \ref{L:calcullocale} and \cite[Proposition 7.4]{rou}.
\end{proof}

\begin{corollary}\label{C:boundRouquierDim}
Let $X$ be a reduced quasi--projective curve over some base field $\kk$. Then there is the following upper bound
on the Rouquier dimension of $D^b\bigl(\Coh(X)\bigr)$:
\begin{equation}\label{E:boundRouquierDim}
\mathsf{dim}\Bigl(D^b\bigl(\Coh(X)\bigr)\Bigr) \le \mathsf{min}(n+1, d +2),
\end{equation}
where $n$ is the level of $X$ and $d$ is the global dimension of the quasi--hereditary algebra $Q$ associated with
$X$.
\end{corollary}

\begin{remark}
In the case $X$ is rational with only simple nodes or  cusps   as singularities, the bound (\ref{E:boundRouquierDim}) has been obtained in \cite[Theorem 10]{bd}. Note that $n = 1$ and $d=0$ is this case. We do not know whether
the estimates (\ref{E:1stEstderDim}) and (\ref{E:2ndEstderDim}) are strict.
\end{remark}

\noindent
The following result gives an affirmative answer on a question posed to the first--named author by Valery Lunts.
\begin{theorem}\label{T:tilting}
For any reduced rational projective curve $X$  over some base field $\kk$ there exists a finite dimensional
quasi--hereditary $\kk$--algebra $\Lambda$ having the following properties.
\begin{itemize}
\item There exists a fully faithful exact functor
$
\Perf(X) \stackrel{\sI}\lar D^b(\Lambda-\mathsf{mod})
$ and a Verdier localization $D^b(\Lambda-\mathsf{mod})
  \stackrel{\sP}\lar D^b\bigl(\Coh(X)\bigr),
$
such that $\sP \sI \simeq \mathsf{Id}_{\Perf(X)}$.
\item The triangulated category $D^b(\Lambda-\mathsf{mod})$ is a recollement of the triangulated categories $D^b\bigl(\Coh(\widetilde{X})\bigr)$ and $D^b(Q-\mathsf{mod})$, where $Q$ is the quasi--hereditary algebra associated with $X$.
\item We have: $\mathrm{gl.dim}(\Lambda) \le d +2$, where $d = \mathrm{gl.dim}(Q)$.
\end{itemize}
\end{theorem}

\begin{proof} According to Theorem \ref{T:main}, there exists a fully faithful exact functor
$\Perf(X) \stackrel{\sL\sF}\lar D^b\bigl(\Coh(\bbX)\bigr)$ and a Verdier localization  $D^b\bigl(\Coh(\bbX)\bigr) \stackrel{\sD\sG}\lar D^b\bigl(\Coh(X)\bigr)$ such that $\sD\sG \cdot \sL\sF \simeq \mathsf{Id}_{\Perf(X)}$.
It suffices to show that the derived category $D^b\bigl(\Coh(\bbX)\bigr)$ has a tilting object. Recall  that we have constructed a semi--orthogonal decomposition
\begin{equation}\label{sorth}
D^b\bigl(\Coh(\bbX)\bigr) = \bigl\langle \langle\kQ\rangle,\,\mathsf{Im}(\sL\tilde\sF)\bigr\rangle,
\end{equation}
where $\langle \kQ\rangle \cong D^b(Q-\mathsf{mod})$ is the triangulated subcategory generated by
$\kQ = \kA/\kJ$ and $\mathsf{Im}(\sL\tilde\sF) \cong D^b\bigl(\Coh(\widetilde{X})\bigr)$.

Since the curve $X$ is rational and projective, we have: $\widetilde{X} = \widetilde{X}_1 \cup \dots \cup
\widetilde{X}_t$, where $\widetilde{X}_k \cong \mathbbm{P}^1_{\kk}$ for all $1 \le k \le t$.
Then $$\widetilde\kB := \bigl(\kO_{\widetilde{X}_1}(-1) \oplus \kO_{\widetilde{X}_1}\bigr) \oplus \dots \oplus
\bigl(\kO_{\widetilde{X}_t}(-1) \oplus \kO_{\widetilde{X}_t}\bigr)$$ is a tilting bundle on $\widetilde{X}$ and the algebra  $E:= \bigl(\End_{\widetilde{X}}(\widetilde\kB)\bigr)^{\mathrm{op}}$ is isomorphic to the direct product of $t$ copies of the path algebra of the Kronecker quiver $\kron\bu\bu$.
Then $\kB:= \sF(\widetilde\kB) \cong \sL\sF(\widetilde\kB)$ is a tilting object in
the triangulated category $\mathsf{Im}(\sL\tilde\sF)$.

The semi--orthogonal decomposition \eqref{sorth} implies  that $\Hom_{D^b(\bbX)}(\kY, \kX) = 0$ for  any $\kX \in \langle \kQ\rangle$ and $\kY \in \mathsf{Im}(\sL\tilde\sF)$.

It is clear that $\Ext_{\bbX}^p(\kQ, \kQ) = 0$ for $p \ge 1$ and $Q \cong \End_{\bbX}(\kQ)^{\mathrm{op}}$.
Since the ideal $\kJ$ is locally projective
as a left $\kA$--module, we have: $\kExt_{\bbX}^p(\kQ,\,-\,) = 0$ for $p \ge 2$. Moreover, since
$\kB$ is locally projective and $\kQ$ is torsion, we also have vanishing $\kHom_{\bbX}(\kQ, \kB) = 0$.
Since $\Hom_{\bbX}(\kX_1, \kX_2) \cong \Gamma\bigl(X, \kHom_{\bbX}(\kX_1, \kX_2)\bigr)$ the local--to--global spectral sequence implies that
$
\Ext^p_{\bbX}(\kQ, \kB) = 0
$
unless $p = 1$ and \begin{equation}\label{E:local-to-global}
\Ext^1_{\bbX}(\kQ, \kB) \cong  \Gamma\bigl(X, \kExt^1_{\bbX}(\kQ, \kB)\bigr).
\end{equation}
Summing up, the complex $\kH := \kQ[-1] \oplus \kB$ is \emph{tilting} in the derived category $D^b\bigl(\Coh(\bbX)\bigr)$. A result of Keller \cite{kel} implies that the derived categories
$D^b\bigl(\Coh(\bbX)\bigr)$ and $D^b(\Lambda-\mathsf{mod})$ are equivalent, where $\Lambda := \bigl(\End_{D^b(\bbX)}(\kH)\bigr)^{\mathrm{op}}$. Finally, observe that
$
\Lambda
\cong
\left(
\begin{array}{cc}
Q & W \\
0 & E
\end{array}
\right),
$
where $W := \Ext^1_{\bbX}(\kQ, \kB)$ viewed as a ($Q$--$E$)--bimodule. Since the algebra $Q$ is quasi--hereditary and $E$ is directed, the algebra $\Lambda$ is quasi--hereditary as well. According to \cite[Corollary 4']{pr}, we have:
$\mathrm{gl.dim}(\Lambda) \le \mathrm{gl.dim}(Q) + 2$.
\end{proof}

\begin{remark} In a recent work \cite[Theorem 4.10]{Wei}, the following inversion of Theorem \ref{T:tilting} was obtained. Assume  $X$ is a projective curve over an algebraically closed field $\kk$ and $\Lambda$ a finite dimensional $\kk$--algebra of finite global dimension such that there exist functors
$$\Perf(X) \stackrel{\sI}\lar D^b(\Lambda-\mathsf{mod}) \stackrel{\sP}\lar D^b\bigl(\Coh(X)\bigr)
$$
  with $\sI$  fully faithful, $\sP$ essentially surjective and $\sP \sI \simeq \mathsf{Id}$. Then $X$ is rational. This result can be shown by examining the Grothendieck groups of the involved triangulated categories.
\end{remark}

\begin{remark}
In the case $X$ has only simple nodes or cusps as singularities, Theorem \ref{T:tilting} has been obtained
in \cite[Theorem 9]{bd}. See also \cite[Definition 3]{bd} for an explicit description of the algebra $\Lambda$ is this case.
\end{remark}

\begin{remark}\label{R:bimodule}
Now we outline how the $Q$--$E$--bimodule $W  = \Ext^1_{\bbX}(\kQ, \kB)$ from the proof of Theorem \ref{T:tilting} can be explicitly determined. The isomorphism
(\ref{E:local-to-global}) implies that $W$ can be computed locally and we may assume that $X = \Spec(O)$ and
$O$ is a complete local ring. We follow the notation of Section \ref{Sec:LocalStory}. For any $1 \le k \le n$ the
left $A$--module $R_k:= Q e_k$ has projective resolution
$$
0 \lar Je_k \lar A e_k \lar R_k \lar 0.
$$
This yields the following  isomorphisms of $\widetilde{O}$--modules:
\begin{equation}\label{E:localcomputation}
W_k = \Ext^1_A(R_k, P) \cong  \dfrac{\Hom_A(J e_k, A e)}{\Hom_A(A e_k, A e)} \cong \dfrac{\Hom_O(C_k, \widetilde{O})}{\Hom_O(O_k, \widetilde{O})} \cong  \dfrac{C_k^\vee}{\widetilde{O}},
\end{equation}
where $P = A e$ and  $C_k = \Hom_{O}(\widetilde{O}, O_k) = \Hom_{O_k}(\widetilde{O}, O_k)$ is the conductor ideal  of the
partial normalization
$O_k \subset \widetilde{O}$. Since  $\widetilde{O}$ is regular, we have a (non--canonical) isomorphism of
$\widetilde{O}$--modules $\dfrac{C_k^\vee}{\widetilde{O}} \cong \dfrac{\widetilde{O}}{C_k}$.
Since $\widetilde{O} \cong \End_A(P)$, this leads to a description of the right $E$--action on $W$. To say more about the left action of $Q$ on $W$, we need an explicit description of the algebra $Q$.
\end{remark}

\section{Quasi--hereditary algebras associated with simple curve singularities}
Let $\kk$ be an algebraically closed field of characteristic zero.
In this section we compute the algebra $\kQ$  for the
 \emph{simple plane curve singularities} in the sense of Arnold \cite{avg}. These singularities are in
 one--to--one correspondence with the simply laced Dynkin graphs.

\begin{proposition}
The algebra $Q$ associated with  the simple singularity $O = \kk\llbracket u, v\rrbracket/(u^2 - v^{m+1})$
of type $A_{m}$ is the path algebra of  the following quiver
\begin{align*}
      &\xymatrix{ 1 \ar@/_/[r]_{\be_1}& 2\ar@/_/[l]_{\al_1} \ar@/_/[r]_{\be_2}  &
  	3 \ar@/_/[l]_{\al_2} \ar@{.}[r] &  (n-1) & n \ar@/_/[l]+<4ex,.9ex>*{}_-{\al_{n-1}}
  	\ar@{<-}@/^/[l]+<4ex,-.9ex>*{}^-{\be_{n-1}}	}
 \end{align*}
 where $n=\big[\frac{m+1}2\big]$ with the relations
 \begin{align*}
   	&\ \be_k\al_k=\al_{k+1}\be_{k+1}\ \text{ if }\ 1\le k<n-1,\\
  	&\ \be_{n-1}\al_{n-1}=0.
 \end{align*}
 $\mathsf{gl.dim}(Q) = 0$
  for $m = 1$ and $2$ and $\mathsf{gl.dim}(Q) = 2$ for all $m\ge 3$.
\end{proposition}
\begin{proof} A straightforward computation shows that $O$ has level $n$. Moreover,
$O_1 \oplus \dots \oplus O_{n+1}$ is the additive generator of the category of maximal Cohen--Macaulay modules, see \cite[Section 7]{Bass}, \cite[Section 13.3]{lw} or
\cite[Section 9]{Yoshino}. It clear that $Q = \kk$ for $m=1$ and $2$.
For $m \ge 3$ we obtain a description of $Q$ in terms of a quiver with relations just taking
first the \emph{Auslander--Reiten quiver} of the category of maximal Cohen--Macaualay $O$--modules subject to the mesh relations (see again \cite[Section 13.3]{lw} or
\cite[Section 9]{Yoshino}), and then deleting the vertex (or  two vertices,  depending whether $m$ is odd or even) corresponding
to the normalization $O_{n+1}$.

The minimal projective resolutions of the simple $Q$-modules $\,U_k$ corresponding to the $k$-th vertex  are:
 \begin{align*}
  &0\to P_2 \xarr{\be_1} P_1 \to U_1\to0,\\
  &0\to P_k \xarr{\smtr{\al_k\\-\be_{k-1}}} P_{k+1}\oplus P_{k-1}
  	\xarr{\smtr{\be_k&\al_{k-1}}} P_k\to U_k\to0 \ \text{ if }\ 1<k<n,\\
  &0\to P_n \xarr{\be_{n-1}} P_{n-1}\xarr{\al_{n-1}} P_n \to U_n\to 0\\
 \end{align*}
Therefore, $\mathsf{gl.dim}(Q) = 2$ for $m \ge 3$ as claimed.
\end{proof}

\begin{remark}\label{R:compA2n}
Assume $O = \kk\llbracket u, v\rrbracket/(u^2 - v^{2n+1})$. Then
$O \cong  \kk\llbracket t^2, t^{2n+1}\rrbracket$ and in this notation we have:
$O_1 = O$, $O_{n+1} = \widetilde{O} = \kk\llbracket t\rrbracket$ and $O_k = \kk\llbracket t^2, t^{2n-2k+3}\rrbracket$ for $1 \le k \le n$.
The morphism $O_k \stackrel{\beta_k}\lar  O_{k+1}$ is identified with the canonical embedding and
 $O_{k+1} \stackrel{\alpha_k}\lar  O_{k}$ is given by the multiplication with $t^2$. The $k$-th conductor ideal $C_k =
 \Hom_{O}(\widetilde{O}, O_k)$ has the following description: $C_k = t^{2(n-k+1)}\cdot \widetilde{O}$.
 Now we can give a full description of the bimodule $W = \Ext^1_A(Q, P)$ from Remark \ref{R:bimodule}.
 \begin{itemize}
 \item As a (right) $\widetilde{O}$--module, it has a decomposition
 \begin{align*}
 W &  \cong  \Ext^1_A(R_1, P) \oplus \Ext^1_A(R_2, P) \oplus \dots \oplus
   \Ext^1_A(R_n, P) \\ & \cong \kk\llbracket t\rrbracket/(t^{2n}) \oplus \kk\llbracket t\rrbracket/(t^{2n-2}) \oplus \dots \oplus
 \kk\llbracket t\rrbracket/(t^{2}),
 \end{align*}
 where $R_k = Q e_k$ for $1 \le k \le n$.
 \item However, as  a left $Q$--module, $W$  is generated just by two elements
 $\gamma_1, \gamma_2\in \Ext^1_A(R_1, P)$  satisfying the following relations:
 $$
 \gamma_1 t = \gamma_2 \quad \mbox{and} \quad \gamma_2 t = \alpha_1 \beta_1 \gamma_1.
 $$
 For $n =1$ (simple cusp) the last relation has to be understood as $\gamma_2 t = 0$ since
 $\alpha_1 \beta_1 = 0$ in this case.
 \end{itemize}
 Assume now $X$ is rational, irreducible  and projective with a singular point $p \in X$ of type $A_{2n}$. Let
 $\mathbbm{P}^1 = \widetilde{X} \stackrel{\nu}\lar X$ be the normalization of $X$ and $\nu^{-1}(p) = (0: 1)$
 with respect to the homogeneous coordinates $z_0, z_\infty \in \Gamma\bigl(\mathbbm{P}^1, \kO_{\mathbbm{P}^1}(1)\bigr)$.
 Then in the algebra $\Lambda$ from Theorem \ref{T:tilting} we have the following relations:
$$
  \gamma_1 z_0 = \gamma_2 z_\infty \quad \mbox{and} \quad
  \gamma_2 z_0 =\alpha_1\beta_1\gamma_1 z_\infty.
$$
Again, for $n = 1$ the last relation has to be understood as $\gamma_2 z_0 = 0$, what is consistent with \cite[Definition 3]{bd}. \qed
\end{remark}

Omitting the details, we state now the descriptions of the algebra $Q$ for $D_m$ and $E_l$ singularities ($m \ge 4$ and $l = 6,7$ or $8$).

\begin{proposition} Let $O = \kk\llbracket u,v\rrbracket/(u^2 v - v^{m-1})$. Then $O$ has level $n = \big[\frac{m}{2}\big]$ and
the quasi--hereditary algebra $Q$ is isomorphic to the path algebra of the following quiver
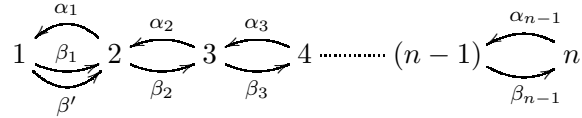
\[
    \xymatrix{  1\ar@/_/[r]^{\be_1} \ar@/_1pc/[r]_{\be'} &
    2\ar@/_/[r]_{\be_2} \ar@/_1pc/[l]_{\al_1}
    & 3\ar@/_/[r]_{\be_3} \ar@/_/[l]_{\al_2} &
  	4 \ar@/_/[l]_{\al_3} \ar@{.}[r] & (n-1) & n \ar@/_/[l]+<4ex,.9ex>*{}_{\al_{n-1}}
  	\ar@{<-}@/^/[l]+<4ex,-.9ex>*{}^{\be_{n-1}} }
 \]
 with the relations
  \begin{align*} 	&\ \be_k\al_k=\al_{k+1}\be_{k+1}\ \text{ if }\ 1\le k<n-1,\\
  	&\ \be_{n-1}\al_{n-1}=0,\\
  	&\ \be'\al_1=0,\\
  	&\ \be_2\be'=0.
 \end{align*}
 We have: $\mathsf{gl.dim}(Q) = 2$ if $n =2$ (i.e.~for types $D_4$ and $D_5$) and $\mathsf{gl.dim}(Q) = 3$ for $n \ge 3$.
 \end{proposition}

 \begin{proposition}\label{P:typeE}
The $E_6$--singularity $\kk\llbracket u,v\rrbracket/(u^3  + v^4)$ has level two and the associated algebra $Q$
is given
 by the quiver with relations
 \[
  \xymatrix{  1\ar@/_/[r]^{\be} \ar@/_1pc/[r]_{\be'} &
    2 \ar@/_1pc/[l]_{\al} } \qquad \be\al=\be'\al=0.
 \]
Its global dimension is equal to $2$.
The $E_7$--singularity $\kk\llbracket u,v\rrbracket/(u^3  + u v^3)$ and $E_8$--singularity $\kk\llbracket u,v\rrbracket/(u^3  + v^5)$
have both  level $3$. In both cases,  associated algebra $Q$ is given by the quiver with relations
 \[
   \xymatrix{  1\ar@/_/[r]^{\be_1} \ar@/_1pc/[r]_{\be'} &
    2\ar@/_/[r]_{\be_2} \ar@/_1pc/[l]_{\al_1}
    & 3 \ar@/_/[l]_{\al_2} } \qquad  \be_1\al_1=\al_2\be_2,\ \be_2\al_2=\be_2\be'=\be'\al_1=0.
 \]
 and its global dimension is equal to $3$.
\end{proposition}
\begin{remark}
  The algebras from Proposition \ref{P:typeE} coincide with those for $D_m$, where $m=4$ or $5$ for $E_6$ and $m=6$ or $7$
  for $E_7$ and $E_8$.
\end{remark}

\begin{example}
Let $X$ be a rational projective curve with two irreducible components $X_1$ and $X_2$ and three
 singular points $x_1\in X_1$ of type $E_6$,
 $x_2\in X_1\cap X_2$ of type $D_7$ and $x_3\in X_2$ of type $A_5$.
Proceeding as explained  in Remark \ref{R:bimodule}  and outlined in Remark \ref{R:compA2n}, we conclude that
 the quiver of the algebra
 $\Lambda$ from Theorem \ref{T:tilting} is
 \[
  \xymatrix@R=1em{ &&& 1_1 \ar@/^1em/[r]^{\al_1} & 2_1 \ar@/^/[l]_{\be_1'} \ar@/^1em/[l]^{\be_1} \\
  -1_1 \ar@/^/[r]^{\xi_1} \ar@/_/[r]_{\eta_1} & 0_1 \ar@/^1em/[urr]|{\ga_{11}} \ar[urr]|{\ga_{12}}
  \ar@/_1em/[urr]|{\ga_{13}} \ar[drr]|{\ga_{23}} \\
  &&& 1_2 \ar@/^1em/[r]^{\al_{21}} & 2_2 \ar@/^/[l]_{\be_{21}'} \ar@/^1em/[l]^{\be_{21}} \ar@/^/[r]^{\al_{22}}
  & 3_2  \ar@/^/[l]^{\be_{22}} \\
  -1_2 \ar@/^/[r]^{\xi_2} \ar@/_/[r]_{\eta_2} & 0_2 \ar@/^/[urr]|{\ga_{21}} \ar@/_/[urr]|{\ga_{22}}
  \ar@/_/[drr]|{\ga_{31}} \ar@/^/[drr]|{\ga_{32}} \\
  &&& 1_3 \ar@/^/[r]^{\al_{31}} & 2_3 \ar@/^/[l]^{\be_{31}} \ar@/^/[r]^{\al_{32}}
  & 3_3  \ar@/^/[l]^{\be_{32}}
  }
 \]
 \end{example}

\end{document}